\pdfoutput=1

\documentclass[12pt]{amsart}
\usepackage{epsfig,color}
\usepackage{blindtext}
\usepackage{graphicx}
\usepackage{enumitem}
\usepackage{url}
\usepackage{amssymb}
\usepackage{graphicx,import}
\usepackage{comment}
\usepackage{esint}
\usepackage{xcolor}
\usepackage{mathtools}
\usepackage{comment}
\usepackage{tikz}

\usepackage[margin = 1in] {geometry}

\usepackage{hyperref}
\usepackage{dsfont}

\newtheorem{theorem}{Theorem}[section]

\newtheorem{lemma}[theorem]{Lemma}
\newtheorem{corollary}[theorem]{Corollary}

\theoremstyle{definition}
\newtheorem{example}[theorem]{Example}
\newtheorem{definition}[theorem]{Definition}
\newtheorem*{remark*}{Remark}
\newtheorem*{acknowledgements}{Acknowledgements}
\newtheorem{remark}[theorem]{Remark}

\numberwithin{equation}{section}

\newcommand{\cA}{\mathcal{A}}

\newcommand{\cC}{\mathcal{C}}
\newcommand{\cD}{\mathcal{D}}
\newcommand{\cH}{\mathcal{H}}

\newcommand{\cM}{\mathcal{M}}
\newcommand{\cS}{\mathcal{S}}

\newcommand{\R}{\mathbb{R}}

\newcommand{\N}{\mathbb{N}}
\newcommand{\Z}{\mathbb{Z}}

\newcommand{\M}{\mathbb{M}}

\newcommand{\bF}{{\mathbf F}}
\newcommand{\bL}{{\mathbf L}}
\newcommand{\Span}{\mathrm{span}}

\DeclareMathOperator{\spt}{spt}

\DeclareMathOperator{\area}{area}
\DeclareMathOperator{\width}{width}
\DeclareMathOperator{\ind}{index}
\newcommand{\eps}{\varepsilon}

\title{Localization of min-max minimal hypersurfaces}

\author{Douglas Stryker}
\address{Department of Mathematics, Princeton University, Princeton, NJ 08540, USA}
\email{dstryker@princeton.edu}

\begin{document}

\maketitle

\begin{abstract}
We show that min-max minimal hypersurfaces can be localized. As a consequence, we obtain the sharp generalization to complete manifolds of the famous Almgren-Pitts min-max theorem in closed manifolds. We use this result to prove the existence of a complete embedded finite area minimal hypersurface of index at most one in every balanced complete manifold.
\end{abstract}

\section{Introduction}
The existence of an abundance of closed minimal hypersurfaces in any closed Riemannian manifold has been established as the culmination of a Morse theory framework for the area functional, called \emph{min-max} (we refer the reader to \cite{MN_morse, MN_morse_mult, MN_pos_ricci, LMN_weyl, IMN_dense, MNS_equi, zhou_mult, song_inf}). This framework additionally provides bounds on the area and the Morse index of the produced minimal hypersurfaces. However, it is difficult to deduce any additional information about these hypersurfaces.

One interesting feature of min-max minimal hypersurfaces that has received substantial attention is their topology. For surfaces in 3-manifolds, a modified Morse theory framework was developed in \cite{smith} (see also \cite{CDL, DLP, ketover_genus}), which yields an upper bound on the genus of the resulting minimal surface. We refer the reader to \cite{Ketover, 4spheres} for exciting recent developments in this direction. Unfortunately, this modified framework does not work in higher dimensions, due to limitations of the required regularity theory.

In this paper, we focus on a different and somewhat overlooked feature of min-max minimal hypersurfaces: their location in the ambient space. While potentially useful in a variety of problems, we emphasize that localization of min-max minimal hypersurfaces is \emph{essential} to prove the existence of minimal hypersurfaces in \emph{non-compact} ambient spaces. Indeed, min-max sequences may run off to infinity in general without some form of localization.

\subsection{Previous work}
There have been three particularly relevant results towards localization of min-max minimal hypersurfaces in general ambient spaces\footnote{\cite{song_inf} also studies the localization of minimal hypersurfaces in a stronger sense (i.e. constraining the hypersurface the lie fully within a given set) in manifolds with cylindrical ends. However, this result crucially uses the rigid structure of these manifolds.}.

\begin{itemize}
	\item \cite{montezuma}: If $(M, g)$ admits a bounded open subset $U$ with strictly mean concave boundary (and $(M, g)$ satisfies some additional hypotheses about its asymptotic geometry), then \cite{montezuma} uses min-max to produce a closed embedded minimal hypersurface intersecting $U$.
	\item \cite{CL}: If $(M, g)$ admits a bounded open subset $U$ with the property that a certain notion of min-max width (for 1-parameter sweepouts) is at least ten times the area of the boundary of $U$, then \cite{CL} uses min-max to produce an embedded complete minimal hypersurface of finite area intersecting any open neighborhood of $\overline{U}$. Notably, this result suffices to prove the existence of a complete minimal hypersurface of finite area in any complete manifold of finite volume, generalizing to higher dimensions the analogous result for geodesics due to \cite{thorbergsson, bangert}.
	\item \cite{song_dich}: If $(M, g)$ admits a compact domain $K$ that has no singular strictly mean-convex foliation, then \cite{song_dich} uses min-max and the mean curvature flow to produce an embedded complete minimal hypersurface of finite area and index at most one intersecting $K$\footnote{See also \cite{gromov} for an earlier proof of a similar result.}. This result suffices to prove the existence of a complete minimal hypersurface of finite area and \emph{index at most one} in any complete manifold of finite volume, improving on the result of \cite{CL}.
\end{itemize}

From the perspective of general existence problems, criteria related to the nontriviality of some min-max invariant seems to be the most natural and useful. One reason is that criteria related to the curvature of certain submanifolds can be hard to verify. Moreover, the \emph{triviality} of certain min-max invariants can be a useful hypothesis in proving existence by other means (i.e.\ Lyusternik-Schnirelmann theory, see \cite{LS}), enabling proofs of existence under weaker hypotheses by a dichotomy argument (for example, see the dichotomy set up in our proof of Theorem \ref{thm:balanced}).

In the ``min-max invariant'' framework, several questions remain regarding the localization of min-max minimal hypersurfaces (including several questions posed in \cite{CL}).
\begin{itemize}
	\item \emph{Does the same result hold if the min-max width is only assumed to be greater than the area of the boundary of $U$?} (\cite[\S2.5 Question 1]{CL}.)
	\item \emph{Does the same result hold for a minimal hypersurface intersecting $\overline{U}$?} (\cite[\S2.5 Question 3]{CL}.)
	\item \emph{Does the same result hold for a minimal hypersurface with index at most 1?} (\cite[\S 2.5 Question 2]{CL}.)
	\item \emph{Does an analogous result hold for min-max using higher parameter sweepouts?}
\end{itemize}

We emphasize the none of the previous work in this area deals with higher parameter min-max. While the modest applications we consider in this paper only use 1-parameter min-max constructions, we expect that localization for higher parameter sweepouts will be useful to prove the existence of more than one minimal hypersurface in some non-compact settings.

\subsection{Min-max localization}
The following result is the sharp analogue in complete manifolds of the Almgren-Pitts min-max theorem (see \cite{pitts}). Moreover, it answers all of the previous questions affirmatively\footnote{We note that we use a different formulation of min-max from \cite{CL}. For a comparison of the min-max frameworks considered in \cite{CL} versus here, see \S\ref{sec:comp}. We emphasize that our proof will work in any min-max framework for which the multiplicity one result of \cite{zhou_mult} applies--we only use this formulation for simplicity, as it is precisely the formulation used in \cite{zhou_mult}.}. We refer the reader to \S\ref{sec:setup} for the precise min-max definitions.

\begin{theorem}\label{thm:main-loc}
Let $(M^{n+1}, g)$ be complete of dimension $3 \leq n+1 \leq 7$. Let $X^k \subset I^m$ be a $k$-dimensional cubical complex, and let $Z \subset X$ be a cubical subcomplex. Let $\Phi_0 : X \to \cC(M)$ be an $\bF$-continuous map. If
\[ \bL(\Phi_0, M) > \sup_{z \in Z} \cH^n(\partial \Phi_0(z)), \]
then there exists a smooth complete (possibly non-compact) embedded minimal hypersurface $\Sigma \subset M$ satisfying
\begin{itemize}
	\item $\cH^n(\Sigma) \leq \bL(\Phi_0, \Span(\Phi_0))$,
	\item $\ind(\Sigma) \leq k$,
	\item $\Sigma \cap \Span(\Phi_0) \neq \varnothing$, where $\Span(\Phi_0) := \overline{\bigcup_{x \in X} \partial \Phi_0(x)}$.
\end{itemize}
\end{theorem}

The localization aspect of this result (i.e.\ nonempty intersection with the span) is new even for closed ambient spaces, and we expect that there may be interesting applications in that setting alone. Moreover, we prove the analogous statement in compact manifolds with strictly mean convex boundary in the course of the proof of Theorem \ref{thm:main-loc}.

\subsection{Balanced manifolds}
As a modest application of this theory, we find that Theorem \ref{thm:main-loc} applies to a large class of complete manifolds defined only by their asymptotic behavior. Roughly, we say that a complete manifold is \emph{balanced} if we can decompose $M = E \sqcup K \sqcup F$ for $K$ compact so that the asymptotic minimal area of cross sections of $E$ and $F$ agree; namely
\begin{align*}
\lim_{r \to \infty} & \inf\{\cH^n(\Sigma) \mid \Sigma \text{\ is\ homologous\ to\ } \partial E\ \text{in\ }E,\ \Sigma \cap B_r(K) = \varnothing\}\\
& = \lim_{r \to \infty} \inf\{\cH^n(\Sigma) \mid \Sigma \text{\ is\ homologous\ to\ } \partial F\ \text{in\ }F,\ \Sigma \cap B_r(K) = \varnothing\}.
\end{align*}
We refer to \S\ref{sec:balanced} for the precise definition.

\begin{theorem}\label{thm:main-balanced}
Every complete balanced $(M^{n+1}, g)$ of dimension $3 \leq n+1 \leq 7$ admits a smooth complete embedded minimal hypersurface of finite area and index at most 1.
\end{theorem}

In \S\ref{sec:balanced}, we present an example to show the necessity of the balanced condition.

We emphasize that Theorem \ref{thm:main-balanced} applies to manifolds exhausted by compact subsets with smooth boundaries whose area tends to zero, which includes finite volume manifolds. Hence, we provide a simpler proof of \cite[Corollary 1.3]{CL} and \cite[Corollary 2.2 (1)]{song_dich}:

\begin{corollary}
Every complete finite volume $(M^{n+1}, g)$ of dimension $3 \leq n+1 \leq 7$ admits a smooth complete embedded minimal hypersurface of finite area and index at most 1.
\end{corollary}

In most cases (see Remark \ref{rem:comp}), Theorem \ref{thm:main-balanced} is completely new, as it requires the sharpness of our localization result.

\subsection{Sketch of the proof of Theorem \ref{thm:main-loc}}
As compared to the careful cut-and-paste constructions of \cite{CL}, our proof is extremely geometrically simple. The ideas here somewhat resemble the ``cutting along hypersurfaces'' arguments in \cite{song_dich}, but our proof involves many fewer complications and tools. For example, we do not require the use of the level set flow.

We first reduce to the case of compact manifolds with smooth, compact, strictly mean convex boundary by modifying the metric near the boundary of a compact exhaustion of $M$ and using the compactness theory of \cite{sharp}.

We further reduce to the case where the metric is bumpy, again using the compactness theory of \cite{sharp}.

By the resolution of the multiplicity one conjecture due to \cite{zhou_mult} in bumpy metrics and the quantitative minimality of strictly stable minimal hypersurfaces due to \cite{local_minmax} (see also \cite{MN_morse_mult}), min-max will produce a connected \emph{unstable} 2-sided smooth closed embedded minimal hypersurface $\Sigma$ with bounded area and bounded index. Suppose that $\Sigma$ does not intersect $\Span(\Phi_0)$. Since each side of a small tubular neighborhood of a connected 2-sided unstable minimal hypersurface has a strictly mean convex foliation, there is a subdomain $W \subset M$ disjoint from a component of $\Sigma$ and containing $\Span(\Phi_0)$ that has strictly mean convex boundary, so we can now restart the argument in $W$. Since there are only finitely many smooth closed embedded minimal hypersurfaces satisfying the area and index bound (by bumpiness), we produce a hypersurface that intersects $\Span(\Phi_0)$ after only finitely many iterations of this argument. We refer the reader to Figure \ref{fig:main} for an illustration of the argument.

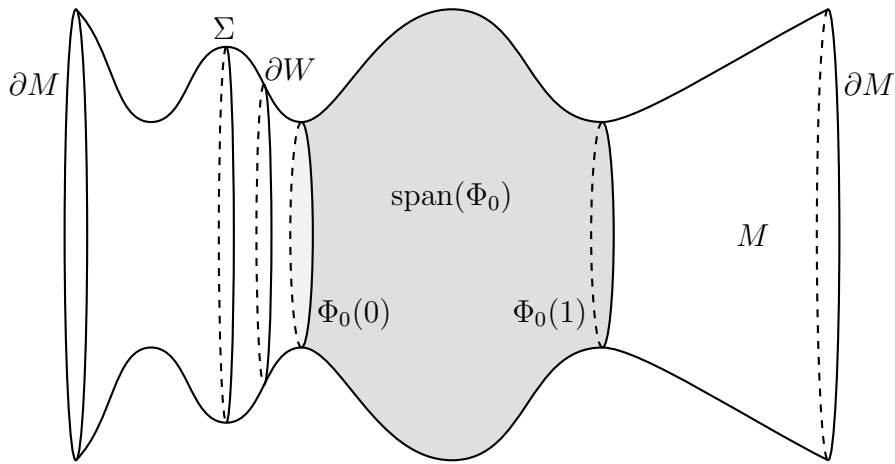
\begin{figure}
\centering
\begin{tikzpicture}

	\fill[gray!25!white] (-2, 1.5) arc [start angle=90, end angle=-90, x radius=.15, y radius=1.5]
		.. controls (-1.5, -1.5) and (-1, -3) .. (0, -3)
		.. controls (1, -3) and (1, -1.5) .. (2, -1.5)
		arc [start angle=-90, end angle=90, x radius=.15, y radius=1.5]
		.. controls (1, 1.5) and (1, 3) .. (0, 3)
		.. controls (-1, 3) and (-1.5, 1.5) .. (-2, 1.5);
	
	\fill[gray!10!white] (-2, 0) ellipse (.15 and 1.5);
	
	\draw[thick] (-5, 3) .. controls (-4.5, 2.5) and (-4.5, 1.5) .. (-4, 1.5)
		.. controls (-3.5, 1.5) and (-3.5, 2.5) .. (-3, 2.5)
		.. controls (-2.5, 2.5) and (-2.5, 1.5) .. (-2, 1.5)
		.. controls (-1.5, 1.5) and (-1, 3) .. (0, 3)
		.. controls (1, 3) and (1, 1.5) .. (2, 1.5)
		.. controls (2.5, 1.5) and (4, 2.5) .. (5, 3);
	
	\draw[thick] (-5, -3) .. controls (-4.5, -2.5) and (-4.5, -1.5) .. (-4, -1.5)
		.. controls (-3.5, -1.5) and (-3.5, -2.5) .. (-3, -2.5)
		.. controls (-2.5, -2.5) and (-2.5, -1.5) .. (-2, -1.5)
		.. controls (-1.5, -1.5) and (-1, -3) .. (0, -3)
		.. controls (1, -3) and (1, -1.5) .. (2, -1.5)
		.. controls (2.5, -1.5) and (4, -2.5) .. (5, -3);
		
	\draw[thick] (-5, 0) ellipse (.15 and 3);
	
	\draw[thick] (5, 3) arc [start angle=90, end angle=-90, x radius=.15, y radius=3];
	\draw[thick, dashed] (5, 3) arc [start angle=90, end angle=270, x radius=.15, y radius=3];
	
	\draw[thick] (-3, 2.5) arc [start angle=90, end angle=-90, x radius=.1, y radius=2.5];
	\draw[thick, dashed] (-3, 2.5) arc [start angle=90, end angle=270, x radius=.1, y radius=2.5];
	
	\draw[thick] (-2.5, 2) arc [start angle=90, end angle=-90, x radius=.1, y radius=2];
	\draw[thick, dashed] (-2.5, 2) arc [start angle=90, end angle=270, x radius=.1, y radius=2];
	
	\draw[thick] (-2, 1.5) arc [start angle=90, end angle=-90, x radius=.15, y radius=1.5];
	\draw[thick, dashed] (-2, 1.5) arc [start angle=90, end angle=270, x radius=.15, y radius=1.5];
	
	\draw[thick] (2, 1.5) arc [start angle=90, end angle=-90, x radius=.15, y radius=1.5];
	\draw[thick, dashed] (2, 1.5) arc [start angle=90, end angle=270, x radius=.15, y radius=1.5];

	\draw (0, .5) node {$\Span(\Phi_0)$};
	
	\draw (4, 0) node {$M$};
	
	\draw (-1.3, -1.05) node {$\Phi_0(0)$};
	
	\draw (1.3, -1.05) node {$\Phi_0(1)$};
	
	\draw (-3.03, 2.75) node {$\Sigma$};
	
	\draw (-2.15, 2.2) node {$\partial W$};
	
	\draw (-5.55, 2) node {$\partial M$};
	
	\draw (5.55, 2) node {$\partial M$};

\end{tikzpicture}
\caption{An illustration of the proof of Theorem \ref{thm:main-loc}.}
\label{fig:main}
\end{figure}

\acknowledgements I am grateful to my advisor Fernando Cod{\'a} Marques for his support and for suggesting the topic of min-max in non-compact spaces. I am also grateful to Lorenzo Sarnataro and Otis Chodosh for providing feedback on earlier versions of this paper.

I was supported by an NDSEG Fellowship.

\section{Setup and Notation}\label{sec:setup}
Let $(M^{n+1}, g)$ be a smooth Riemannian manifold, either closed, complete non-compact, or compact with smooth compact boundary.

\begin{definition}[$\cC(M)$]
We define the appropriate class of finite perimeter sets in each case:
\begin{itemize}
\item If $(M, g)$ is closed, we define $\cC(M)$ to be the set of finite perimeter sets $\Omega \subset M$.
\item If $(M, g)$ is complete non-compact, we define $\cC(M)$ to be the set of finite perimeter sets $\Omega \subset M$ so that $\partial \Omega$ is bounded and nonempty. We emphasize that $\Omega \in \cC(M)$ need not be bounded; we only require that $\partial \Omega$ is bounded.
\item If $(M, g)$ is compact with smooth compact boundary, we define $\cC(M)$ to be the set of finite perimeter sets $\Omega \subset M$ so that $\partial \Omega \cap \mathrm{Int}(M) \neq \varnothing$ and $\overline{\partial \Omega \cap \mathrm{Int}(M)} \cap \partial M = \varnothing$. In other words, $\partial \Omega \cap \mathrm{Int}(M)$ is nonempty and contained in a compact subset of $\mathrm{Int}(M)$.
\end{itemize}
We note that in any case, if $W \subset M$ is a bounded open domain with smooth boundary, $\Omega \in \cC(M)$, and $\partial \Omega \subset W$, then $\Omega \cap W \in \cC(W)$.
\end{definition}

Henceforth, we follow the setup and notation of \cite{zhou_mult}.

\begin{definition}[$\bF$-metric on $\cC(M)$]
For $\Omega \in \cC(M)$, we let $[\Omega]$ denote the integral $(n+1)$-current associated to $\Omega$, and we let $|\partial \Omega|$ denote the varifold associated to $\partial \Omega$. For $\Omega_1,\ \Omega_2 \in \cC(M)$, we define
\[ \bF(\Omega_1, \Omega_2) := \mathcal{F}([\Omega_1] - [\Omega_2]) + \bF(|\partial \Omega_1|, |\partial \Omega_2|), \]
where $\mathcal{F}$ is the flat norm on currents and $\bF$ is the usual $\bF$-metric on varifolds.
\end{definition}

Let $X^k \subset I^m$ be a $k$-dimensional cubical complex, and let $Z \subset X$ be a cubical subcomplex. Let $\Phi_0 : X \to \cC(M)$ be a continuous map in the $\bF$ topology.

\begin{definition}[Span]
The \emph{span} of $\Phi_0$ is the set
\[ \Span(\Phi_0) := \overline{\bigcup_{x\in X} \partial \Phi_0(x)}. \]
By definition, $\Span(\Phi_0) \subset \mathrm{Int}(M)$ is compact.
\end{definition}

\begin{definition}[Homotopy class]
Let $\Pi(\Phi_0, M)$ be the set of all sequences of $\bF$-continuous maps $\{\Phi_i : X \to \cC(M)\}_{i\in \N}$ with the property that there are flat-continuous homotopies $\{\Psi_i : [0, 1] \times X \to \cC(M)\}_{i \in \N}$ satisfying $\Psi_i(0, \cdot) = \Phi_0$, $\Psi_i(1, \cdot) = \Phi_i$, and
\[ \limsup_{i\to \infty}\sup\{\bF(\Psi_i(s, z), \Phi_0(z)) \mid s \in [0, 1],\ z\in Z\} = 0. \]
$\Pi(\Phi_0, M)$ is called the \emph{homotopy class of $\Phi_0$ in $M$}.
\end{definition}

\begin{definition}[Homotopy class on subdomain]
If $W \subset M$ is a bounded open domain with smooth boundary and $\Span(\Phi_0) \subset W$, then we define
\[ \Pi(\Phi_0, W) := \Pi(\Phi_0 \cap W, W), \]
where $(\Phi_0 \cap W)(x) := \Phi_0(x) \cap W$. $\Pi(\Phi_0, W)$ is called the \emph{homotopy class of $\Phi_0$ in $W$}.
\end{definition}

\begin{definition}[Width]
The \emph{width} of $\Pi(\Phi_0, W)$ is
\[ \bL(\Phi_0, W) := \inf_{\{\Phi_i\} \in \Pi(\Phi_0, W)} \limsup_{i \to \infty} \sup_{x \in X} \cH^n(\partial \Phi_i(x)). \]
\end{definition}

\begin{remark}
We emphasize the ambient space in the definition because we must consider restrictions to subdomains. It follows from a straightforward extension argument that if $\Span(\Phi_0) \subset W_1 \subset W_2 \subset M$ are bounded open domains with smooth boundary, then
\[ \bL(\Phi_0, W_1) \geq \bL(\Phi_0, W_2). \]
\end{remark}

\begin{definition}[Width on subset]
If $\Span(\Phi_0) \subset C \subset M$ is an arbitrary subset, we define
\[ \bL(\Phi_0, C) := \sup\{\bL(\Phi_0, W) \mid W\supset C \text{\ bounded open domain with smooth boundary}\}, \]
and the supremum is finite because $\bL(\Phi_0, W) \leq \sup_{x \in X} \cH^n(\partial \Phi_0(x)) < \infty$ for any bounded open domain with smooth boundary $W$ containing $\Span(\Phi_0)$.

In particular, this definition allows us to make sense of $\bL(\Phi_0, \Span(\Phi_0))$.
\end{definition}

\section{Localization in Compact Manifolds}
Suppose $(M^{n+1}, g)$ is a smooth compact Riemannian manifold of dimension $3 \leq n+1 \leq 7$ with smooth compact (possibly empty) boundary.

It will be convenient to approximate arbitrary metrics by \emph{bumpy} metrics.

\begin{definition}
The metric $g$ is \emph{weakly bumpy} if every smooth closed embedded minimal hypersurface in $M$ is non-degenerate.
\end{definition}

\begin{lemma}\label{lem:generic}
The set of weakly bumpy metrics on $M$ is $C^{\infty}$ dense.
\end{lemma}
\begin{proof}
Let $g$ be any metric on $M$. By a standard extension argument, there is a smooth closed Riemannian manifold $(N, \overline{g})$ that has $(M, g)$ as a subset. By White's bumpy metric theorem \cite[Theorem 2.2]{white_bumpy}, there is a sequence of bumpy metrics on $N$ converging smoothly to $\overline{g}$. The restrictions of these metrics to $M$ are weakly bumpy by definition, and they converge smoothly to $g$.
\end{proof}

We collect facts about min-max from many foundational works in the following lemma.

\begin{lemma}\label{lem:min-max}
Suppose the metric $g$ is weakly bumpy and $\partial M$ is strictly mean convex. If
\[ \bL(\Phi_0, M) > \sup_{z \in Z} \cH^n(\partial \Phi_0(z)), \]
then there is a 2-sided smooth closed embedded minimal hypersurface $\Sigma \subset M$ satisfying
\begin{itemize}
	\item $\cH^n(\Sigma) = \bL(\Phi_0, M)$,
	\item $1 \leq \ind(\Sigma) \leq k$.
\end{itemize}
\end{lemma}
\begin{proof}
We first review the proof in the closed case (i.e. $\partial M = \varnothing$). By \cite[Theorem 4.1]{zhou_mult}, there is a 2-sided smooth closed embedded minimal hypersurface $\Sigma \subset M$ satisfying $\cH^n(\Sigma) = \bL(\Phi_0, M)$ and $\ind(\Sigma) \leq k$. Since $\Sigma$ is nondegenerate, \cite[Theorem 1.1]{local_minmax} (see also \cite[Theorem 6.1]{MN_morse_mult}) implies that $\Sigma$ cannot be stable, so $\ind(\Sigma) \geq 1$.

In the case of strictly mean convex boundary, we observe as in \cite{wang_obstacle} (noting that by definition $\partial \Omega \cap \mathrm{Int}(M)$ for $\Omega \in \cC(M)$ is a cycle in $\mathrm{Int}(M)$) that we can perform the pull-tight procedure in the min-max construction using only isotopies generated by vector fields pointing into $M$ to produce constrained stationary varifolds, which are strictly contained in the interior of $M$ by the strict mean convexity of $\partial M$ and the strong maximum principle (see \cite[Theorem 7]{white_maximum}). The regularity theory, index lower bound, index upper bound, and approximation by PMC hypersurfaces are all local arguments and still apply in this setting\footnote{The observation that these local arguments extend to manifolds with strictly mean convex boundary can be seen in the proof of \cite[Theorem 6.6]{song_dich}.}. We note that we must choose the approximating prescribing function to be small relative the the positive lower bound on the mean curvature of $\partial M$ so that we can constrain the PMC solutions to the interior of $M$. There is no problem because we ultimately take the prescribing function to zero uniformly in the proof of \cite[Theorem 4.1]{zhou_mult}.
\end{proof}

The essential point of Lemma \ref{lem:min-max} is that the min-max hypersurface is \emph{unstable}. We can now remove the bumpiness assumption and argue by approximation.

\begin{theorem}\label{thm:compact-localization}
Suppose $\partial M$ is strictly mean convex. If
\[ \bL(\Phi_0, M) > \sup_{z\in Z} \cH^n(\partial \Phi_0(z)), \]
then there exists a smooth closed embedded minimal hypersurface $\Sigma \subset M$ satisfying
\begin{itemize}
	\item $\cH^n(\Sigma) \leq \bL(\Phi_0, \Span(\Phi_0))$,
	\item $\ind(\Sigma) \leq k$,
	\item $\Sigma \cap \Span(\Phi_0) \neq \varnothing$.
\end{itemize}
\end{theorem}
\begin{proof}
We begin by assuming that the metric $g$ is weakly bumpy.

Let $\cM$ denote the set of connected 2-sided smooth closed embedded minimal hypersurfaces $\Sigma \subset M$ satisfying
\begin{itemize}
	\item $\cH^n(\Sigma) \leq \bL(\Phi_0, \Span(\Phi_0))$,
	\item $1 \leq \ind(\Sigma) \leq k$.
\end{itemize}

By Sharp's compactness theorem \cite[Theorem 2.3]{sharp} $\cM$ is a finite set.

By Lemma \ref{lem:min-max} applied to $\bL(\Phi_0, M)$, $\cM \neq \varnothing$.

Suppose for contradiction that $\Sigma \cap \Span(\Phi_0) = \varnothing$ for all $\Sigma \in \cM$.

Let $\cD \subset \cM$ be a maximal disjoint subset; namely, (1) if $\Sigma, \Sigma' \in \cD$ and $\Sigma \cap \Sigma' \neq \varnothing$, then $\Sigma = \Sigma'$, and (2) every $\Sigma \in \cM$ satisfies $\Sigma \cap \Sigma' \neq \varnothing$ for some $\Sigma' \in \cD$. Such a set $\cD$ exists by induction.

Since each $\Sigma \in \cD$ is 2-sided and has index at least 1, there is a mean convex foliation (with curvature pointing away from $\Sigma$) in a neighborhood of $\Sigma$ obtained by flowing $\Sigma$ in the normal direction according to the positive first eigenfunction of the Jacobi operator (see \cite[Proof of Lemma 3.2]{HK}). Hence, by deleting a small neighborhood of each $\Sigma \in \cD$ (and then deleting any resulting components that do not contain $\Span(\Phi_0)$), there is an open set $W \subset M$ satisfying
\begin{itemize}
	\item $\Span(\Phi_0) \subset W$,
	\item $W$ is disjoint from $\Sigma$ for all $\Sigma \in \cD$,
	\item $W$ has smooth, strictly mean convex boundary.
\end{itemize}

By Lemma \ref{lem:min-max} for $\bL(\Phi_0, W)$, there is a smooth closed embedded minimal hypersurface $\Sigma^* \subset W$ so that $\Sigma^* \in \cM$, contradicting the defining property of $\cD$.

By Sharp's compactness theorem \cite[Theorem A.6]{sharp} for changing ambient metrics and Lemma \ref{lem:generic}, the conclusion now follows for arbitrary metrics.
\end{proof}

\section{Localization in Complete Non-Compact Manifolds}

Suppose $(M^{n+1}, g)$ is a smooth complete non-compact Riemannian manifold of dimension $3 \leq n+1 \leq 7$.

\begin{theorem}\label{thm:noncompact-localization}
If
\[ \bL(\Phi_0, M) > \sup_{z \in Z} \cH^n(\partial \Phi_0(z)), \]
then there exists a smooth complete (possibly non-compact) embedded minimal hypersurface $\Sigma \subset M$ satisfying
\begin{itemize}
	\item $\cH^n(\Sigma) \leq \bL(\Phi_0, \Span(\Phi_0))$,
	\item $\ind(\Sigma) \leq k$,
	\item $\Sigma \cap \Span(\Phi_0) \neq \varnothing$.
\end{itemize}
\end{theorem}
\begin{proof}
Let $\{M_i\}_{i \in \N}$ be an exhaustion of $M$ by nested precompact open subsets with smooth boundary satisfying $\Span(\Phi_0) \subset M_1$.

We construct a modified metric $\tilde{g}_i$ on $M_i$ so that
\begin{itemize}
	\item $\tilde{g}_i$ agrees with $g$ outside $B_1(\partial M_i) \setminus B_1(\Span(\Phi_0))$,
	\item $\tilde{g}_i \geq g$ in the sense of bilinear forms,
	\item $\partial M_i$ has strictly mean convex boundary in $\tilde{g}_i$.
\end{itemize}
Namely, using Fermi coordinates in a small tubular neighborhood of $\partial M_i$, we can smoothly transition from the metric $g$ to a metric on the cylinder $(0, 1) \times \partial M_i$ that is pointwise larger than $g$ and so that $\{0\} \times M_i$ has strictly mean convex boundary.

Since $\tilde{g}_i \geq g$, we have $\bL(\Phi_0, M_i, \tilde{g}_i) \geq \bL(\Phi_0, M_i, g) \geq \bL(\Phi_0, M, g)$. Moreover, since $\tilde{g}_i = g$ on $B_1(\Span(\Phi_0))$, we have $\bL(\Phi_0, M_i, \tilde{g}_i) \leq \bL(\Phi_0, \Span(\Phi_0), \tilde{g}_i) = \bL(\Phi_0, \Span(\Phi_0), g)$.

By Theorem \ref{thm:compact-localization}, there is a smooth closed embedded minimal hypersurface $\Sigma_i \subset M_i$ (with respect to $\tilde{g}_i$) satisfying
\begin{itemize}
	\item $\cH^n(\Sigma_i) \leq \bL(\Phi_0, \Span(\Phi_0), g)$,
	\item $\ind_{\tilde{g}_i}(\Sigma) \leq k$,
	\item $\Sigma_i \cap \Span(\Phi_0) = \varnothing$.
\end{itemize}

By Sharp's compactness theorem \cite[Theorem 2.3]{sharp}\footnote{While the statement of \cite[Theorem 2.3]{sharp} concerns closed hypersurfaces in closed manifolds, the arguments are local and apply on compact subsets of complete manifolds.}, the sequence $\{\Sigma_i\}$ converges locally smoothly away from at most $k$ points to a smooth complete (not necessarily compact) embedded minimal hypersurface $\Sigma \subset M$ satisfying the conclusions of the theorem.
\end{proof}

\subsection{Comparison to \cite{CL}}\label{sec:comp}
We begin by observing that we use the Almgren-Pitts min-max framework as in \cite{zhou_mult} (for example), whereas \cite{CL} uses the constructions of \cite{delellis_tasnady}. Since the equivalence of these frameworks is besides the point, we make comparisons by analogy only.

Let $(M^{n+1}, g)$ be a smooth complete Riemannian manifold, and let $U \subset M$ be a bounded open set with smooth boundary.

In \cite{CL}, the \emph{width} of $U$, denoted $W(U)$, is defined to be the min-max width over families of hypersurfaces $\{\Sigma_t = \partial \Omega_t\}_{t\in [0, 1]}$ as in \cite{delellis_tasnady} so that $U \cap \Omega_0 = \varnothing$ and $U \subset \Omega_1$. The \emph{relative width} of $U$, denoted $W_{\partial}(U)$, is defined to be the min-max width over families as above that are furthermore nested, and where we only compute the area of each hypersurface inside $U$. Ultimately, \cite[Theorem 8.2]{CL} shows that for any $\delta > 0$, there is a smooth complete (not necessarily compact) embedded minimal hypersurface of area at most $W_{\partial}(U) + \cH^n(\partial U)$ intersecting the $\delta$-neighborhood of $U$ so long as
\[ W_{\partial}(U) \geq 10\cH^n(\partial U). \]

In our framework, let $\{\Sigma_t = \partial \Omega_t\}$ be a nested family of hypersurfaces satisfying $\Sigma_0 \sqcup \Sigma_1 = \partial U$. We can view this family as a map $\Phi_0 : [0, 1] \to \cC(M)$, which satisfies $\Span(\Phi_0) = \overline{U}$. We show that there is a smooth complete (not necessarily compact) embedded minimal hypersurface of area at most $\bL(\Phi_0, \overline{U})$ and index at most 1 intersecting $\overline{U}$ so long as
\[ \bL(\Phi_0, M) > \max\{\cH^n(\Sigma_0), \cH^n(\Sigma_1)\}. \]

We draw the analogy between the two frameworks.
\begin{itemize}
	\item $\bL(\Phi_0, M)$ is analogous to $W(U)$, and $W(U) \geq W_{\partial}(U)$.
	\item $\bL(\Phi_0, \overline{U})$ is analogous to a quantity bounded from above by $W_{\partial}(U) + \cH^n(\partial U)$. 
	\item $\max\{\cH^n(\Sigma_0), \cH^n(\Sigma_1)\} \leq \cH^n(\partial U)$.
\end{itemize}
Hence, our result has weaker hypotheses and stronger conclusions than \cite{CL}.

\begin{remark}
The result of \cite{CL} applies when the dimension is $n+1 \geq 3$, with the usual optimal non-smooth regularity when $n+1 > 7$. We emphasize that our proof only works when $3 \leq n+1 \leq 7$, which are the dimensions where \cite{CL} produces a \emph{smooth} minimal hypersurface.
\end{remark}

\section{Existence in Balanced Manifolds}\label{sec:balanced}
Let $(M^{n+1}, g)$ be a smooth complete non-compact Riemannian manifold of dimension $3 \leq n+1 \leq 7$.

Let $K \subset M$ be a compact set with smooth compact boundary. Let $E \subset M \setminus K$ be the disjoint union of some of the components of $M \setminus K$ (where we allow $E = \varnothing$ and $E = M \setminus K$). We let $F := (M\setminus K) \setminus E$.

Let $\alpha \in H_n(M, \Z/2\Z)$ be the homology class of $\partial E$ in $M$ (which is the same as the homology class of $\partial F$ in $M$). Let $\alpha_E \in H_n(E, \Z/2\Z)$ be the homology class of $\partial E$ in $E$. Let $\alpha_F \in H_n(F, \Z/2\Z)$ be the homology class of $\partial F$ in $F$.

Fix some $x_0 \in K$.

We define
\[ \cA_E := \lim_{r \to \infty} \inf \{\M(T) \mid T \in \alpha_E \mathrm{\ and\ } \spt(T) \subset E \setminus B_r(x_0) \}, \]
\[ \cA_F := \lim_{r \to \infty} \inf \{\M(T) \mid T \in \alpha_F \mathrm{\ and\ } \spt(T) \subset F \setminus B_r(x_0) \}. \]
For convenience, we let $\cA_E$ (resp. $\cA_F$) equal 0 if $E$ (resp. $F$) is bounded.

We say $(M, g)$ is \emph{balanced} if there is a choice of $K$, $x_0$, $E$, and $F$ so that $\cA_E = \cA_F$.

\begin{theorem}\label{thm:balanced}
Every complete balanced $(M^{n+1}, g)$ of dimension $3 \leq n+1 \leq 7$ admits a smooth complete (not necessarily compact) embedded minimal hypersurface of finite area and index at most 1.
\end{theorem}

\begin{remark}
If $(M, g)$ is furthermore \emph{thick at infinity}, meaning that any connected finite area complete minimal hypersurface in $(M, g)$ is closed (see \cite{song_dich}), then the minimal embedding is closed.
\end{remark}

\begin{remark}\label{rem:comp}
We emphasize that Theorem \ref{thm:balanced} is completely new when $\cA_E = \cA_F > 0$, as the localization results of \cite{CL} and \cite{song_dich} are insufficient to handle these cases.
\end{remark}

\subsection{Examples}
We present two simple classes of manifolds that satisfy the hypotheses of Theorem \ref{thm:balanced}.

\begin{example}[Finite Volume Manifolds]
Suppose $(M^{n+1}, g)$ has finite volume. We can define $K = \overline{B}_{2\eps}(x) \setminus B_{\eps}(x)$ for some $x \in M$ and $\eps < \mathrm{inj}(M, p)$, and let $E = B_\eps(x)$ and $F = M \setminus \overline{B}_{2\eps}(x)$. Since $E$ is bounded, we have $\cA_E = 0$. Since $M$ has finite volume, $\cA_F = 0$.

In fact, the same argument applies to any manifold admitting an exhaustion by compact sets with smooth boundaries whose areas tend to zero\footnote{In fact, the proof in this case follows immediately from the relative isoperimetric inequality in a compact subset with smooth boundary and Theorem \ref{thm:main-loc}.}, as in \cite{song_dich}.
\end{example}

\begin{example}[Balanced Asymptotically Cylindrical Manifolds]
Let $(M^{n+1}, g)$ be \emph{asymptotically cylindrical}. Namely, there is a compact subset $K \subset M$, a smooth closed (not necessarily connected) Riemannian manifold $(N^n, h)$, and a diffeomorphism $\phi: \R_+ \times N \to M \setminus K$ so that
\[ \limsup_{r \to \infty} \|(dt^2 \oplus h) - \phi^*g\|_{C^{\infty}((r, \infty) \times N)} = 0. \]
These manifolds generalize the manifolds with cylindrical ends considered by \cite{song_inf}, although they need not be rigid on any finite scale.

If there is a decomposition of $N = R \sqcup S$ so that each of $R$ and $S$ is the disjoint union of components of $N$ and $\cH^n(R) = \cH^n(S) = \frac{1}{2}\cH^n(N)$, then $(M, g)$ is balanced and Theorem \ref{thm:balanced} applies.
\end{example}

\begin{example}[Unbalanced]
The balanced condition in Theorem \ref{thm:balanced} is necessary, as demonstrated by the following example. Let $(\R \times S^n, g)$ be a warped product metric given by
\[ g = dt^2 \oplus f(t)g_{S^n}, \]
where $g_{S^n}$ is the round metric on $S^n$ and $f : \R \to (1/2, 2)$ is a strictly decreasing smooth function. Then $(M, g)$ has a strictly mean convex foliation and bounded geometry, so $(M, g)$ does not admit any finite area embedded minimal hypersurfaces. Importantly, since $\lim_{t \to -\infty} f(t) > \lim_{t \to \infty} f(t)$, $(M, g)$ is not balanced.
\end{example}

\subsection{Proof of Theorem \ref{thm:balanced}}
We define two geometric invariants.

\begin{definition}
The \emph{area} of the homology class $\alpha \in H_n(M, \Z/2\Z)$ is
\[ \area(\alpha) := \inf_{T \in \alpha} \M(T). \]
\end{definition}

\begin{definition}
Let
\[ \cA_E^r := \inf\{\M(T) \mid T \in \alpha_E \text{\ and\ } \spt(T) \subset E \setminus B_r(x_0)\}, \]
\[ \cA_F^r := \inf\{\M(T) \mid T \in \alpha_F \text{\ and\ } \spt(T) \subset F \setminus B_r(x_0)\}. \]
Let $\cS_r$ denote the set of $\bF$-continuous maps $\Phi : [0, 1] \to \cC(M)$ so that\footnote{In the case that (without loss of generality) $E$ is bounded, we instead require that $\partial \Phi(0) \subset E$ and satisfies $\cH^n(\partial \Phi(0)) \leq 1/r$.}
\begin{itemize}
	\item $\partial\Phi(t) \in \alpha$ for all $t\in [0, 1]$,
	\item $\partial \Phi(0) \subset E \setminus B_r(x_0)$ and satisfies $\cH^n(\partial \Phi(0)) \leq \cA_E^r + \frac{1}{r}$,
	\item $\partial \Phi(1) \subset F \setminus B_r(x_0)$ and satisfies $\cH^n(\partial \Phi(1)) \leq \cA_F^r + \frac{1}{r}$.
\end{itemize}
The \emph{width} of the homology class $\alpha \in H_n(M, \Z/2\Z)$ is
\[ \width(\alpha) := \limsup_{r \to \infty} \sup_{\Phi \in \cS_r} \bL(\Phi, M). \]
\end{definition}

Note that we always have $\area(\alpha) \leq \width(\alpha)$.

\begin{lemma}\label{lem:noncompact_min}
If $\area(\alpha) < \min\{\cA_E, \cA_F\}$, then there is a smooth complete (not necessarily compact) embedded minimal hypersurface $\Sigma \subset M$ satisfying
\begin{itemize}
	\item $\cH^n(\Sigma) \leq \area(\alpha)$,
	\item $\ind(\Sigma) = 0$.
\end{itemize}
\end{lemma}
\begin{proof}
Let $\{\Omega_i\}_{i \in \N}$ be an exhaustion of $M$ by precompact connected nested open sets with smooth boundary. We suppose without loss of generality that $K \subset \Omega_i$ for all $i$.

We define
\[ \area_i(\alpha) := \inf_{T \in \alpha \mid \spt(T) \subset \overline{\Omega}_i} \M(T). \]
Since the sets $\Omega_i$ are nested, we have $\area_i(\alpha) \geq \area_j(\alpha)$ if $i < j$. Since $\Omega_i$ is an exhaustion of $M$ and each $T \in \alpha$ has compact support, we have
\[ \lim_{i \to \infty} \area_i(\alpha) = \area(\alpha). \]

By the standard compactness theory for currents (see \cite[Theorem 6.8.2]{simon}), there is a current $T_i^* \in \alpha$ satisfying $\spt(T_i^*) \subset \overline{\Omega}_i$ and $\M(T_i^*) = \area_i(\alpha)$. Since $T_i^*$ is area minimizing in $\Omega_i$, $T_i^*\llcorner \Omega_i$ is the current of a smooth stable embedded minimal hypersurface with integer multiplicity (see \cite[Theorem 7.5.8]{simon}).

Since $\min\{\cA_E, \cA_F\} > \area(\alpha)$, there is a compact set $C \subset M$ so that $\spt(T_i^*) \cap C \neq \varnothing$ for all $i$ sufficiently large.

By the curvature estimates for stable minimal immersions \cite[Corollary 1]{schoen_simon}, there is a subsequence of $\{T_i^*\}_{i \in \N}$ that converges locally smoothly to a nontrivial smooth complete (not necessarily closed) stable embedded minimal hypersurface $\Sigma \subset M$ satisfying $\cH^n(\Sigma) \leq \area(\alpha)$.
\end{proof}

\begin{lemma}\label{lem:noncompact_minmax}
If $\width(\alpha) > \max\{\cA_E, \cA_F\}$, then there is a smooth complete (not necessarily compact) embedded minimal hypersurface $\Sigma \subset M$ satisfying
\begin{itemize}
	\item $\cH^n(\Sigma) < \infty$,
	\item $\ind(\Sigma) \leq 1$.
\end{itemize}
\end{lemma}
\begin{proof}
By assumption, there is an $\bF$-continuous map $\Phi: [0,1] \to \cC(M)$ satisfying
\[ \width(\alpha) \geq \bL(\Phi, M) > \max\{\cH^n(\partial \Phi(0)), \cH^n(\partial \Phi(1))\}. \]
Hence, Theorem \ref{thm:noncompact-localization} applies, and the conclusion follows.
\end{proof}

\begin{lemma}\label{lem:noncompact_rigidity}
If $\area(\alpha) = \width(\alpha)$, then for every $x \in M$ there is a smooth complete (not necessarily compact) embedded minimal hypersurface $\Sigma_x \subset M$ satisfying
\begin{itemize}
	\item $\cH^n(\Sigma_x) \leq \area(\alpha)$,
	\item $\ind(\Sigma_x) = 0$,
	\item $x \in \Sigma_x$.
\end{itemize}
In particular, there are infinitely many such minimal embeddings.
\end{lemma}
\begin{proof}
Fix $x \in M$. Fix $D \subset M$ a compact set with smooth boundary whose interior contains $x$. By the relative isoperimetric inequality, there is a constant $c > 0$ so that for any $\Omega \in \cC(M)$ satisfying $\cH^{n+1}(\Omega \cap D) = \frac{1}{2}\cH^{n+1}(D)$, we have $\cH^n(\partial \Omega \cap D) \geq c$.

Let $r_j \to \infty$, and let $\Phi^j \in \cS_{r_j}$. Let $\{\Phi_i^j\}_{i \in \N} \in \Pi(\Phi^j, M)$ so that
\[ \limsup_{i \to \infty} \sup_{x\in X} \cH^n(\partial \Phi_i^j(x)) = \bL(\Phi^j, M), \]
which exists by taking a diagonal sequence.
Since $D \subset B_{r_j}(x_0)$ for all $j$ (without loss of generality), there is some $t_i^j \in [0, 1]$ so that $\cH^{n+1}(\Phi_i^j(t_i^j) \cap D) = \frac{1}{2}\cH^{n+1}(D)$. Moreover, by the assumption that $\width(\alpha) = \area(\alpha)$, there is a diagonal sequence so that $\Sigma_j:= \partial \Phi_{i_j}^j(t_{i_j}^j)$ is a minimizing sequence for area in $\alpha$ and satisfies $\cH^n(\Sigma_j \cap D) \geq c$ for all $j$. Hence, there is a smooth complete (not necessarily compact) embedded stable minimal hypersurface $\Sigma \subset M$ satisfying $\cH^n(\Sigma) \leq \area(\alpha)$ and $\Sigma \cap D \neq \varnothing$.

By taking the set $D$ smaller, the compactness theory for stable minimal hypersurfaces with uniform area bounds (see \cite[Corollary 1]{schoen_simon}) produces the desired minimal hypersurface.
\end{proof}

Since $\cA_E = \cA_F$ for balanced manifolds, either Lemma \ref{lem:noncompact_rigidity} or one of Lemma \ref{lem:noncompact_min} or \ref{lem:noncompact_minmax} applies, so Theorem \ref{thm:balanced} follows.

\bibliographystyle{amsalpha}
\bibliography{localization.bib}

\end{document}